\documentclass[a4paper,12pt]{article}
\usepackage{authblk}
\usepackage[a4paper, total={6in, 10in}]{geometry}
\usepackage[utf8]{inputenc}
\usepackage{amsfonts}
\usepackage{amsthm}
\usepackage{amssymb}
\usepackage{amsmath}
\usepackage{enumerate}
\usepackage{hyperref}
\usepackage[]{lipsum}
\usepackage{graphicx}
\usepackage[ruled,vlined]{algorithm2e}
\usepackage{subcaption}
\usepackage{cite}

\usepackage[labelformat=parens,labelsep=quad,skip=3pt]{caption}
\usepackage{float}
\usepackage{fancyhdr}
\usepackage[export]{adjustbox}
\usepackage{url}

\newtheorem{thm}{Theorem}[section]

\newtheorem{cor}{Corollary}[section]
\newtheorem{conj}{Conjecture}[section]
\newtheorem{lem}{Lemma}[section]
\newtheorem{rem}{Remark}[section]
\title{A lower bound of the energy of non-singular graphs in terms of average degree}

\author[1]{Saieed Akbari\thanks{E-mail addresses: s\_akbari@sharif.edu, dabirian@umich.edu, sghasemi@uh.edu}}
\author[2]{Hossein Dabirian}
\author[3]{S. Mahmood Ghasemi}
\affil[1]{Department of Mathematical Science, Sharif University of Technology, Tehran, Iran}
\affil[2]{Department of Electrical Engineering, University of Michigan, Michigan, USA }
\affil[3]{Department of Mathematics, University of Houston, Texas, USA}

\date{}                   
\begin{document}
\maketitle
\begin{abstract}
    Let $G$ be a graph of order $n$ with adjacency matrix $A(G)$. The \textit{energy} of graph $G$, denoted by $\mathcal{E}(G)$, is defined as the sum of absolute value of eigenvalues of $A(G)$. It was conjectured that if $A(G)$ is non-singular, then $\mathcal{E}(G)\geq\Delta(G)+\delta(G)$. In this paper we propose a stronger conjecture as for $n \geq 5$, $\mathcal{E}(G)\geq n-1+\Bar{d}$, where $\Bar{d}$ is the average degree of $G$. Here, we show that conjecture holds for bipartite graphs, planar graphs and for the graphs with $\Bar{d}\leq n-2\ln n -3$.
\end{abstract}



\section{Introduction}
Throughout this paper, all graphs are simple that is with no loop or multiple edges. Let $G$ be a graph. We denote the vertex set and the edge set of $G$ by $V(G)$ and $E(G)$, respectively. Let $A(G)$ be the adjacency matrix of the graph $G$. A graph $G$ is called \textit{non-singular} if $A(G)$ is non-singular. The \textit{energy} of a graph $G$, denoted by $\mathcal{E}(G)$, is defined as the sum of absolute value of eigenvalues of $A(G)$ introduced for the first time by Gutman in \cite{Gutman}. In recent years, many researchers investigated the concept of the energy of graphs and its bounds, see \cite{Ghodrati}, \cite{dasgut}, \cite{Juan} and \cite{Wagner}. In this article, we are interested in studying some new lower bounds for the energy of non-singular graphs.\\
We denote the eigenvalues of $A(G)$ by $\lambda_1,\lambda_2,\ldots,\lambda_n$ from the largest to the smallest one. The two largest absolute value of eigenvalues are denoted by $\mu_1,\mu_2$. Obviously, $\mu_1=\lambda_1$, however, $\mu_2$ is either $\lambda_2$ or $-\lambda_n$. The maximum, the minimum and the average degree of graph $G$ are denoted by $\Delta (G)$, $\delta (G)$, and $\Bar{d}$  respectively. We call $|V(G)|$ and $|E(G)|$ by \textit{order} and the \textit{size} of $G$, respectively. We denote the path of order $n$ by $P_n$.

\section{Main conjectures and primary results }

A well-known approach to find a lower bound for graph energy whose eigenvalues are non-zero is using the fact that if $x>0$, then $x\geq \ln x + 1$ \cite{das2013improving}. Let $G$ be a graph of order $n$ and size $m$. From this inequality it follows that:
\begin{equation}\label{mjl}
    \begin{split}
        \mathcal{E}(G)=\lambda_1 + \sum_{i=2}^n|\lambda_i| & \geq   n -1 +\lambda_1 + \ln |\lambda_2\lambda_3\cdots\lambda_n|\\
        & = n - 1 + \lambda_1 + \ln|\det A(G)| - \ln\lambda_1.
    \end{split}
\end{equation}
By [Proposition 3.1.2, p. 33] of \cite{Spectraofgraphs}, $\lambda_1\geq \Bar{d}=2m/n\geq \delta(G)$. As a result, if $|\det A(G)|\geq \lambda_1$ or equivalently,  $|\lambda_2\lambda_3\cdots\lambda_n|\geq 1$, as $n-1\geq\Delta(G)$, the following conjecture holds:

\begin{conj}
\cite{akbari2020short} The energy of a non-singular graph $G$ satisfies the following inequality: \[\mathcal{E}(G)\geq \Delta(G) + \delta(G)\]
\label{conj1}
\end{conj}
However, $|\det A(G)|\geq\lambda_1$ does not hold for all graphs. The function $x-\ln x$ is increasing for $x\geq 1$. Replacing $\lambda_1$ with $\Bar{d}$ in the above inequality gives us $n-1+\Bar{d}-\ln\Bar{d}$ as a lower bound for the energy. Hence, if $\delta(G)<\Bar{d}-\ln\Bar{d}$ for a graph $G$, then Conjecture \ref{conj1} is true for $G$. Now, we could introduce another conjecture which is a generalization of Conjecture \ref{conj1}:
\begin{conj}
Let G be a non-singular graph. Then $\mathcal{E}(G)\geq n-1+\Bar{d}$ except for $P_4$ and the following graph:\\
\begin{figure}[H]
    \centering
    \includegraphics[scale=.3]{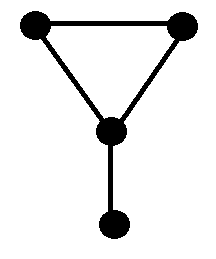}
    \label{con2pic}
\end{figure}

\label{conj2}
\end{conj}
Note that for every graph $G$ of order $n$, $n-1 \geq \Delta(G)$ and $\Bar{d} \geq \delta(G)$, so Conjecture \ref{conj2} is stronger than Conjecture \ref{conj1}.
\begin{rem}
\normalfont Conjecture \ref{conj2} holds for regular graphs. Let $G$ be a $k$-regular graph. Since the characteristic polynomial of $G$ has integer coefficients and is divisible by $\lambda - k$, $k\:|\:\det A(G)$. Hence, $|\det A(G)|\geq k$ and $\mathcal{E}(G)\geq n-1+k = n-1+\Bar{d}$.
\end{rem}

\begin{lem}
If $x \in (0, 7.11]$, then:
\[x\geq\frac{10}{11}+\frac{9}{11}\ln x+ \frac{x^2}{11}\]
\label{lem6.19}
\end{lem}
\begin{proof}
Let $f(x)=x-\frac{10}{11}-\frac{9}{11}\ln x- \frac{x^2}{11}$. We should verify $f(x)>0$ for $\lim_{x\rightarrow 0^+ }$, $\lim_{x\rightarrow 7.11}$, and the critical points, since $f$ is a smooth function on the interval (0, 7.11]. On one hand, $\lim_{x\rightarrow 0^+ }f(x)=+\infty$. On the other hand, $\lim_{x\rightarrow 7.11 }f(x)=f(7.11) \approx 0.0004>0$.\\
For critical points we first compute $f'$:
\[f'(x) = 1 - \frac{9}{11x}-\frac{2x}{11}=\frac{11x - 9 -2x^2}{11x}=\frac{(9-2x)(x-1)}{11x}\]
As a result, $x=1,4.5$ are the only critical points. Finally, $f(1)=0\geq 0$ and $f(4.5)\approx 0.52>0$.
\end{proof}

The next result shows that Conjecture \ref{conj2} holds for graphs with maximum eigenvalue at most 7.11.
\begin{thm}
Let G be a non-singular graph of order $n\geq 5$. If $\lambda_1 \leq 7.11$, then Conjecture \ref{conj2} holds for $G$.
\label{thm6.19}
\end{thm}

\begin{proof}
If $\lambda_1 \leq 7.11$, then $|\lambda_i| \leq 7.11$ for each $i$. Using Lemma \ref{lem6.19}, for each eigenvalue $\lambda_i$, we have,
\[|\lambda_i|\geq \frac{10}{11}+\frac{9}{11}\ln|\lambda_i|+\frac{\lambda_i^2}{11}\]

If $|E(G)| = m$, then sum over $i$ implies that:
\begin{equation*}
    \begin{split}
        \mathcal{E}(G) & = \sum_{i=1}^n |\lambda_i| \\
        \geq & \frac{10n}{11} + \frac{9}{11}\ln|\det A(G)| + \frac{\sum_{i=1}^n \lambda_i^2}{11}\\
        = & \frac{10n}{11} + \frac{9}{11}\ln|\det A(G)| + \frac{2m}{11} \\
        \geq & \frac{10n}{11}+\frac{2m}{11}\ \mathrm{(Since}\ A(G)\ \mathrm{is}\ \mathrm{non}\textnormal{-}\mathrm{singular,} |\det A(G)|\geq 1)\\
        = & n-1+\frac{2m}{n} + \frac{(n-11)(2m-n)}{11n}\\
        = & n-1+\Bar{d}+\frac{(n-11)(2m-n)}{11n}.
    \end{split}
\end{equation*}
For a graph with no isolated vertex, $2m\geq n$. Thus, if $n\geq 11$, then $\mathcal{E}(G)\geq n - 1 + \Bar{d}$. For $5\leq n\leq 10$ we checked the assertion by computer.
\end{proof}

\begin{cor}
Let G be a non-singular graph of order $n$ and size $m$. If $m\leq 2.574n$, then Conjecture \ref{conj2} holds.
\label{cor2.1}
\end{cor}
\begin{proof}
First, note that the function $x-\ln x$ is an increasing function for $x>1$ and $7.11-\ln 7.11\approx 5.1485>5.148$. Hence, $x-\ln x>5.148$, for $x>7.11$. If $\lambda_1 \leq 7.11$, then Theorem \ref{thm6.19} implies the validity of Conjecture \ref{conj2}. Otherwise, $\lambda_1>7.11$ and it follows from the known lower bound for energy that:
\begin{equation*}
    \begin{split}
        \mathcal{E}(G)  \geq & n - 1 + \lambda_1  - \ln\lambda_1 + \ln|\det A(G)| \\
        \geq & n - 1 + \lambda_1 - \ln\lambda_1 \\
        \geq & n - 1 + 5.148 \\
        \geq & n - 1 + \frac{2m}{n}\\
        = & n - 1 + \Bar{d}.
    \end{split}
\end{equation*}

\end{proof}
Now, we are in a position to prove Conjecture \ref{conj2} for bipartite graphs.
\begin{thm}
Conjecture \ref{conj2} holds for bipartite graphs.
\label{thm-bipartite}
\end{thm}
\begin{proof}
Since $G$ is bipartite, by [Proposition 3.4.1, p. 38] of \cite{Spectraofgraphs} $-\lambda_1$ is an eigenvalue of $G$. We also have $\lambda_1 \geq  \Bar{d}$, therefore the following holds:

\begin{equation*}
    \begin{split}
        \mathcal{E}(G)& =  \lambda_1 + |-\lambda_1| + \sum_{i=2}^{n-1}|\lambda_i| \\
        & \geq n - 2 + 2\lambda_1 + \ln|\det A(G)| - 2\ln\lambda_1\\
        & \geq n - 2 + 2\lambda_1 - 2\ln\lambda_1 \\
        & = n - 1 + \lambda_1 + \lambda_1 - 2\ln\lambda_1-1\\
        & \geq n - 1 + \Bar{d} + \lambda_1 - 2\ln\lambda_1-1.
    \end{split}
\end{equation*}
Therefore, it would be enough to show that $\lambda_1 - 2\ln\lambda_1-1>0$. The function $f(x)=x-2\ln x - 1$ is increasing for $x>2$ and $f(7.11)\approx2.19>0$. Hence, for $\lambda_1\geq7.11$ the assertion is proved. For $\lambda_1 \leq 7.11$, by Theorem \ref{thm6.19}, Conjecture \ref{conj2} holds
\end{proof}

\section{A strong lower bound for the energy of non-singular graphs}
Another way to make a lower bound for the graph energy is using AM-GM inequality: 

\begin{equation}\label{ineq-am-gm-simple}
    \begin{split}
        \mathcal{E}(G)=\lambda_1 + \sum_{i=2}^n|\lambda_i| & \geq  \lambda_1 + (n-1) \sqrt[n-1]{ |\lambda_2\lambda_3\cdots\lambda_n|}\\
        & =  \lambda_1 +(n-1) \sqrt[n-1]{\frac{|\det A(G)|}{\lambda_1}}.
    \end{split}
\end{equation}

 Equation (\ref{ineq-am-gm-simple}) is stronger than (\ref{mjl}), since by substituting $x=\sqrt[n-1]{\frac{|\det A(G)|}{\lambda_1}}$ in $x\geq \ln x + 1$, we can get (\ref{mjl}) from (\ref{ineq-am-gm-simple}).  In this section we use a stronger version of AM-GM to improve the bound. First, we point out this version here: 

\begin{thm}
Suppose that $x_1,x_2,\ldots, x_n$ are real positive numbers and $b=\max(x_1,x_2,\ldots,x_n)$. Then, we have: 
\[\frac{x_1+x_2+\cdots+x_n}{n} - \sqrt[n]{x_1x_2\cdots x_n}\geq \frac{1}{2b}\left(\frac{\sum_{i=1}^n x_i^2}{n}-(\frac{\sum_{i=1}^n x_i}{n})^2\right).\]
\label{thm-am-gm-var}
\end{thm}
\begin{proof}
Let $\Bar{x}=\frac{x_1+x_2+\cdots+x_n}{n}$ and set $p_k=\frac{1}{n}$ in  \cite{cartwright1978refinement}. It implies that,
\[\frac{x_1+x_2+\cdots+x_n}{n} - \sqrt[n]{x_1x_2\cdots x_n}\geq \frac{1}{2b}\left(\frac{\sum_{i=1}^n(x_i-\Bar{x})^2}{n}\right).\]

Note that the right hand side is the variance of $x_1,x_2,\ldots,x_n$:
\begin{equation*}
    \begin{split}
        \frac{1}{2b}\left(\frac{\sum_{i=1}^n(x_i-\Bar{x})^2}{n}\right)=&\frac{1}{2bn}\sum_{i=1}^n(x_i^2-2\Bar{x}x_i+\Bar{x}^2)\\
        = & \frac{1}{2bn}(\sum_{i=1}^n x_i^2 - 2\Bar{x}(\sum_{i=1}^n x_i)+n\Bar{x}^2)\\
        = & \frac{1}{2bn}(\sum_{i=1}^n x_i^2 - 2n\Bar{x}^2+n\Bar{x}^2)\\
        = & \frac{1}{2b}\left(\frac{\sum_{i=1}^n x_i^2}{n}-\Bar{x}^2\right)\\
        = & \frac{1}{2b}\left(\frac{\sum_{i=1}^n x_i^2}{n}-(\frac{\sum_{i=1}^n x_i}{n})^2\right).
    \end{split}
\end{equation*}
\end{proof}

\begin{thm}\label{strong-bound}
Let $G$ be a non-singular graph of order $n$ and size $m$. Then the following holds:
\begin{equation}
    \mathcal{E}(G)\geq \mu_1 + (n-1)\left(\sqrt{\mu_2^2+2\mu_2\sqrt[n-1]{\frac{|\det A(G)|}{\mu_1}}+\frac{2m-\mu_1^2}{n-1}}-\mu_2 \right)
\end{equation}

\end{thm}

\begin{proof}
We use Theorem \ref{thm-am-gm-var}. Set $\{x_1,x_2,\ldots, x_{n-1}\} = \{|\lambda_2|,|\lambda_3|,\ldots,|\lambda_n|\}$. For simplicity we define $B_1=\frac{\sum_{i=2}^{n}|\lambda_i|}{n-1}$ and $B_2=\frac{\sum_{i=2}^{n}\lambda_i^2}{n-1}$. Note that $B_1$ is the arithmetic mean of $\{|\lambda_2|,|\lambda_3|,\ldots,|\lambda_n|\}$ and their geometric mean is: $\sqrt[n-1]{|\lambda_2\lambda_3\cdots\lambda_n|}=\sqrt[n-1]{\frac{|\det A(G)|}{\mu_1}}$. As $\lambda_1$ does not appear in this set, we know that $\mu_2=\max\{|\lambda_2|,|\lambda_3|,\ldots,|\lambda_n|\}$.  Using Theorem 3.1:
\[
        B_1 - \sqrt[n-1]{\frac{|\det A(G)|}{\mu_1}} \geq \frac{1}{2\mu_2}(B_2 - B_1^2)=\frac{1}{2\mu_2}(\frac{2m-\mu_1^2}{n-1}-B_1^2).
\]
Therefore,
\[\frac{1}{2\mu_2}B_1^2+B_1-\sqrt[n-1]{\frac{|\det A(G)|}{\mu_1}}-\frac{2m-\mu_1^2}{2\mu_2(n-1)}\geq 0
\]
Solving the resulted quadratic inequality above for $B_1$, we have the following result:
\[B_1\geq  \sqrt{\mu_2^2+2\mu_2\sqrt[n-1]{\frac{|\det A(G)|}{\mu_1}}+\frac{2m-\mu_1^2}{n-1}}-\mu_2.\]
Finally, the equality $\mathcal{E}(G)=\mu_1+(n-1)B_1$ gives the result.
\end{proof}

In Theorem \ref{strong-bound}, let

\[C=\sqrt{\mu_2^2+2\mu_2\sqrt[n-1]{\frac{|\det A(G)|}{\mu_1}}+\frac{2m-\mu_1^2}{n-1}}-\mu_2\]

If $C\geq 1$ then,
\[\mathcal{E}(G)\geq \mu_1+n-1\geq \Bar{d}+n-1\]
and Conjecture \ref{conj2} holds for $G$.\\

Now, let $C\leq 1$.  We are going to refine the lower bound multiplying by its conjugate:
\begin{equation*}
\begin{split}
    \mathcal{E}(G)\geq & \mu_1+(n-1)\left(\sqrt{\mu_2^2+2\mu_2\sqrt[n-1]{\frac{|\det A(G)|}{\mu_1}}+\frac{2m-\mu_1^2}{n-1}}-\mu_2\right)\\
    = & \mu_1 + (n-1)\left(\frac{2\mu_2\sqrt[n-1]{\frac{|\det A(G)|}{\mu_1}}+\frac{2m-\mu_1^2}{n-1}}{C+2\mu_2}\right)
\end{split}
\end{equation*}

Suppose $G$ is non-singular. Then, $|\det A(G)|\geq 1$. Using the inequality $x\geq \ln x + 1$ for $x=\sqrt[n-1]{\frac{|\det A(G)|}{\mu_1}}$:
\begin{equation}
    \begin{split}
       \mathcal{E}(G) \geq & \mu_1 + (n-1)\left(\frac{2\mu_2(1+\frac{\ln|\det A(G)|-\ln\mu_1}{n-1})+\frac{2m-\mu_1^2}{n-1}}{C+2\mu_2}\right)\\
    \geq & \mu_1 + \frac{2\mu_2(n-1)-2\mu_2\ln\mu_1+2m-\mu_1^2}{1+2\mu_2}\\
    = & n-1 + \left(\mu_1 + \frac{2m-\mu_1^2-n+1-2\mu_2\ln\mu_1}{1+2\mu_2}\right)
    \end{split}
    \label{C1}
\end{equation}
This concludes the following corollary:
\begin{cor}
Let $G$ be a non-singular graph. If
\[(\mu_1-\Bar{d})+ (\frac{2m-n+1-\mu_1^2}{2\mu_2+1})\geq\frac{2\mu_2}{2\mu_2+1}\ln\mu_1,\]
then Conjecture \ref{conj2} holds. 
\label{golden-ineq}
\end{cor}

The condition of Corollary \ref{golden-ineq} is interesting. We know that $\Bar{d}\leq\mu_1$. On the other hand, in \cite{yuan1988bound} it was shown that $\mu_1^2\leq 2m-n+1$ for a graph with no isolated vertex. Corollary \ref{golden-ineq} states that in order to prove Conjecture \ref{conj2}, it is sufficient to give an appropriate lower bound on the sum of errors of $\mu_1-\Bar{d}$ and $\frac{2m-n+1-\mu_1^2}{2\mu_2+1}$. We prove Cojecture \ref{conj2} for a wide range of graphs using this idea.
\begin{lem}
Let b,c,d be real numbers and $c\geq 0$. Then, the following function: 
\[f(x)=(2c+1)(x-d)+(b-x^2)-2c(\ln x)\]
is decreasing for $x\in[c,\infty)$.
\label{lem-decreasing}
\end{lem}
\begin{proof}
We have $f'(x)=2c+1-2x-\frac{2c}{x}$. Therefore, the roots of $
f'$ are $\frac{2c+1\pm\sqrt{4c^2-12c+1}}{4}$. As $c\geq0$, then $4c^2-12c+1\leq(2c-1)^2$. Thus, 
\[\frac{2c+1+\sqrt{4c^2-12c+1}}{4}\leq \frac{2c+1+2c-1}{4}=c.\]
As a result, in case of being real, both roots of $f'(x)$ are less than equal $c$. Hence $f'(x)$ is negative for $x\geq c$ i.e. $f$ is decreasing on $[c,\infty)$. 
\end{proof}

The following lemma is easy to prove.
\begin{lem}
Let $x\geq 13$, then $\frac{2(x-1)}{\sqrt{x}} - 4 - \ln x \geq 0$.
\label{ks1}
\end{lem}

\begin{lem}
Let $x\geq13$, then $(x-1)\sqrt{1-\frac{2\ln x + 4}{x}} - x + \ln x + 4\geq 0$.\\

\label{ks2}
\end{lem}
\begin{proof}
We show that: 
\[ (x-1)\sqrt{1-\frac{2\ln x + 4}{x}} \geq x - \ln x - 4.\]
Both sides are positive for $x\geq13$. Therefore, we can square them and this is equivalent to:
\begin{equation*}
    \begin{split}
        (x-1)^2 (\frac{x-2\ln x - 4}{x}) \geq (x - \ln x - 4)^2\\
        2x^2-7x-4-x(\ln x)^2-4x\ln x-2\ln x\geq 0.
    \end{split}
\end{equation*}

It is easy to prove that the last inequality holds for $x\geq 13$. Indeed, one can show $\frac{11x^2}{20}\geq x(\ln x)^2$, $\frac{4x^2}{5}\geq 4x\ln x$, $\frac{x^2}{20}\geq 2\ln x$, and $\frac{12x^2}{20}-7x-4\geq 0$ for $x\geq 13$.
\end{proof}

\begin{thm}
Let $G$ be a non-singular graph of order $n$. If $\Bar{d}\leq n-2\ln n -3$, then Conjecture \ref{conj2} holds for $G$. 
\label{bound-dbar}
\end{thm}
\begin{proof}
By computer we checked the assertion for $n \leq 10$. Thus, we may assume $n \geq 11$. Moreover, we can assume $\Bar{d}\geq 5.14$ due to Corollary \ref{cor2.1}. A straightforward consequence of this is for $n=11,12$, the value of $n-2\ln n -3\approx 3.20,4.03$ and Conjecture \ref{conj2} holds for $\Bar{d}\leq n-2\ln n -3$ in this cases. Hence, we can suppose $n\geq 13$. 
We set $c=\mu_2$, $d=\Bar{d}$, and $b=2m-n+1$ in the Lemma \ref{lem-decreasing}, where $m$ is the size of $G$. The function $f(x)$ becomes: 
\[f(x)=(2\mu_2+1)(x-\Bar{d})+(2m-n+1-x^2)-2\mu_2\ln x\]
which is decreasing for $x\geq\mu_2$. As $\mu_1\geq\mu_2$ and $\sqrt{2m-n+1}\geq\mu_1$,
\begin{equation*}
    \begin{split}
        f(\mu_1)\geq &  f(\sqrt{2m-n+1})\\
        = & (2\mu_2+1)(\sqrt{2m-n+1}-\Bar{d})-2\mu_2\ln(\sqrt{2m-n+1})\\
        = & 2\mu_2\left(\sqrt{2m-n+1} - \Bar{d} - \ln(\sqrt{2m-n+1}))+ (\sqrt{2m-n+1} - \Bar{d}\right)\\
        \geq & 2\mu_2\left(\sqrt{2m-n+1} - \Bar{d} - \ln(\sqrt{2m-n+1})\right)\\
        = & 2\mu_2\left(\sqrt{n\Bar{d}-n+1} - \Bar{d} - \ln(\sqrt{n\Bar{d}-n+1})\right).
    \end{split}
\end{equation*}
Now, we want to show that the parenthesis in the last equation is non-negative for $5\leq\Bar{d}\leq n-2\ln n -3$.\\
First note that $\sqrt{x}-\ln(\sqrt{x})$ is an increasing function for $x\geq 1$, because its derivative is $\frac{1}{2\sqrt{x}}-\frac{1}{2x}$. Hence,
\[\sqrt{n\Bar{d}-n+1} - \Bar{d} - \ln(\sqrt{n\Bar{d}-n+1})\geq \sqrt{n\Bar{d}-n} - \Bar{d} - \ln(\sqrt{n\Bar{d}-n}).\]
For simplicity, we set $ t^2=\frac{\Bar{d}-1}{n}\geq 0$. We have
\begin{equation*}
    \begin{split}
        \sqrt{n\Bar{d}-n} - \Bar{d} - \ln(\sqrt{n\Bar{d}-n})= & \sqrt{n^2t^2} - nt^2 - 1 - \ln(\sqrt{n^2t^2}) \\
        = & nt - nt^2 - \ln n - \ln t-1\\
        \geq &  nt - nt^2 - \ln n - t\\
        = &  (n-1)t - nt^2 - \ln n
    \end{split}
\end{equation*}
This is a quadratic function of $t$. As the coefficient of $t^2$ is negative, this function is concave and we only have to check two sides of the interval to show that it is non-negative. $5\leq\Bar{d}\leq n-2\ln n - 3$ implies $\frac{2}{\sqrt{n}}\leq t\leq \sqrt{1-\frac{2\ln n + 4}{n}}$.
\begin{itemize}
    \item $t=\frac{2}{\sqrt{n}}$.
    \begin{equation*}
    \begin{split}
        (n-1)t - nt^2-\ln n = & \frac{2(n-1)}{\sqrt{n}} - 4 -\ln n \\
        \geq & 0 \quad (\mathrm{Lemma}\; \ref{ks1})
    \end{split}
    \end{equation*}

    \item $t=\sqrt{1-\frac{2\ln n + 4}{n}}$.
    \begin{equation*}
    \begin{split}
        (n-1)t - nt^2-\ln n = & (n-1)\sqrt{1-\frac{2\ln n + 4}{n}} - n + \ln n + 4 \\
        \geq & 0 \quad (\mathrm{Lemma}\;\ref{ks2})
    \end{split}
    \end{equation*}
\end{itemize}
Thus, $f(\mu_1)\geq 0$. In other words: 
\[(2\mu_2+1)(\mu_1-\Bar{d})+(2m-n+1-\mu_1^2)-2\mu_2\ln \mu_1\geq 0,\]
which completes the proof according to Corollary \ref{golden-ineq}.
\end{proof}

\begin{thm}
\label{thm-chromatic}
 If G is a graph with $\chi (G) = 3$, then Conjecture \ref{conj2} holds except for possibly finitely many graphs.
\end{thm}
\begin{proof}

According to Hoffman's Inequality \cite{hoffman1970eigenvalues}, we have $\lambda_n \leq \frac{-\lambda_1}{2}$. Using a similar argument given in the proof of Theorem \ref{thm-bipartite}, since $|\lambda_n| - \ln |\lambda_n| \geq \frac{\lambda_1}{2} - \ln \frac{\lambda_1}{2}$ one can see that:
\begin{equation*}
    \begin{split}
        \mathcal{E}(G)\geq &  \lambda_1+\frac{\lambda_1}{2}+ n-2+\ln\det|A(G)|-\ln\lambda_1-\ln\frac{\lambda1}{2}\\
        \geq & \lambda_1 + \frac{\lambda_1}{2}+n-2-2\ln\lambda_1 +\ln2\\
        \geq & n-1+\Bar{d}+ (\frac{\lambda_1}{2}-2\ln\lambda_1-0.3).
    \end{split}
\end{equation*}
The last parenthesis is positive for $\lambda_1\geq 10$. As a result, Conjecture \ref{conj2} holds for the graphs with chromatic number 3, unless $\lambda_1 \in (7.11, 10)$. We have $\Bar{d} \leq \lambda_1 \leq 10$. For $n \geq 19$, $n -2 \ln n -3 \geq 10$ and so $\Bar{d} \leq n -2 \ln n -3$. Thus Conjecture \ref{conj2} holds for all graphs with chromatic number 3 and order at least 19.
\end{proof}

In Corollary 3 of \cite{akbariplanar}, Conjecture \ref{conj1} was proved for planar graphs. Here we improve this corollary and provide a proof of Conjecture \ref{conj2} for planar graphs.
\begin{thm}
Conjecture \ref{conj2} holds for planar graphs. 
\label{planers}
\end{thm}
\begin{proof}
First note that for planar graphs in [Theorem 6.1.23, p. 241] of \cite{West} it was shown that $m\leq3n-6$, where $m=|E(G)|$ Therefore, $\Bar{d}\leq \frac{6n-12}{n}\leq 6$. If $n\geq 15$, then $n-2\ln n - 3\geq 6$ and Conjecture \ref{conj2} holds as a result of Theorem \ref{bound-dbar}. If $n\leq 14$, then $\Bar{d}\leq 6-\frac{12}{n}\leq 6 - \frac{12}{14}\approx 5.143<5.148$ and Conjecture \ref{conj2} holds due to Theorem \ref{cor2.1}.
\end{proof}


\begin{thm}
If Conjecture \ref{conj2} holds for connected graphs of order $n\geq 5$, then it holds for all graphs of order $n\geq 5$
\end{thm}
\begin{proof}
Suppose a graph $G$ of order $n$ and size $m$ has two components $G_1$ and $G_2$ of orders $n_1$ and $n_2$ and sizes $m_1$ and $m_2$, respectively. If $G$ is non-singular, then $G_1$ and $G_2$ should be also non-singular. If Conjecture \ref{conj2} holds for $G_1$ and $G_2$, we have:
\[\mathcal{E}(G_1)\geq \frac{2m_1}{n_1} + n_1 - 1 \]
\[\mathcal{E}(G_2)\geq \frac{2m_2}{n_2} + n_2 - 1. \]

Without loss of generality, suppose $\frac{2m_1}{n_1}\leq\frac{2m_2}{n_2}$. Hence
$\frac{2m_1}{n_1}\leq \frac{2m_1+2m_2}{n_1+n_2}\leq \frac{2m_2}{n_2}$. Since $G_1$ has no isolated vertices, $\frac{2m_1}{n_1}\geq 1$. Now, we have:
\begin{equation*}
    \begin{split}
        \mathcal{E}(G)=\mathcal{E}(G_1) + \mathcal{E}(G_2) & \geq \frac{2m_2}{n_2}+n_1+n_2-1+ (\frac{2m_1}{n_1} - 1) \\
        & \geq \frac{2m_1+2m_2}{n_1+n_2} + n_1 + n_2 - 1.
    \end{split}
\end{equation*}
Hence, Conjecture \ref{conj2} holds for $G$.\\

There are only 2 non-singular connected graphs whose energy is less than $n-1+\Bar{d}$, $P_4$ and the graph depicted in Conjecture \ref{conj2}. We denote the later graph of order 4 by H. We need to study these cases independently. When $(G_1,G_2)\in\{(P_4,P_4),(P_4,H),(H,H)\}$ the graph $G_1\cup G_2$ is of order 8 and in our computer search, we already checked those cases. Therefore, we have to only show that if Conjecture \ref{conj2} holds for a graph $G_1$ of order $n$ and size $m$, then it will hold for $G_1\cup P_4$ and $G_1\cup H$.

 \begin{itemize}
     \item $G = G_1\cup P_4$. As $\mathcal{E}(P_4)\geq 4.47$, $2m_1\geq n_1\geq 2$:
     \begin{equation*}
         \begin{split}
             \mathcal{E}(G_1\cup P_4)\geq & \mathcal{E}(G_1)+4.47\\
             \geq & \frac{2m_1}{n_1}+ n_1 - 1 + 4.47 \\
             = & \frac{2m_1}{n_1}+ 0.47 + n_1 + 3 \\
             \geq & \frac{2m_1+4}{n_1+4}+ + 0.47 + n_1 + 3 \\
             \geq & \frac{2m_1+6}{n_1+4}+ n_1 + 3 
         \end{split}
     \end{equation*}
     \item $G = G_1\cup H$. As $\mathcal{E}(H)\geq 4.96$, $2m_1\geq n_1\geq 2$:
     \begin{equation*}
         \begin{split}
             \mathcal{E}(G_1\cup H)\geq & \mathcal{E}(G_1)+4.96\\
             \geq & \frac{2m_1}{n_1}+ n_1 - 1 + 4.96 \\
             = & \frac{2m_1}{n_1}+0.96+ n_1 + 3 \\
             \geq & \frac{2m_1+4}{n_1+4}+ +0.96+n_1 + 3\\
             \geq & \frac{2m_1+8}{n_1+4}+ n_1 + 3 
         \end{split}
     \end{equation*}
     
 \end{itemize}
As a result, if $G_1,G_2$ are two arbitrary non-singular connected graphs, then Conjecture \ref{conj2} holds for $G_1\cup G_2$. Now, using induction, the proof is complete for all non-singular graphs of order $n\geq 5$.
\end{proof}


\end{document}